\documentclass[reqno]{amsart}

\usepackage[utf8]{inputenc}
\usepackage{amsmath}
\usepackage{amssymb}
\usepackage{mathrsfs}
\usepackage{graphics}
\usepackage[dvipdf]{graphicx}
\usepackage{geometry}
\usepackage{setspace}
\usepackage{color}
\usepackage{amsthm}
\usepackage{paralist}
\usepackage{dsfont}
\usepackage{verbatim}

 \usepackage[colorlinks=true]{hyperref}
\hypersetup{urlcolor=blue, citecolor=red}

\newtheorem{theorem}{Theorem}[section]

\newtheorem{lemma}[theorem]{Lemma}
\newtheorem{proposition}{Proposition}

\theoremstyle{definition}

\newcommand{\ep}{\varepsilon}

\newcommand{\calH}{{\mathcal{H}}}
\newcommand{\calF}{{\mathcal{F}}}
\newcommand{\W}{\mathbf{W}}

\newcommand{\N}{{\mathbb{N}}}
\newcommand{\R}{{\mathbb{R}}}

\newcommand{\eins}{\mathbf{1}}

\newcommand{\dd}{\,\mathrm{d}}

\newcommand{\dv}{\mathrm{div}}

\newcommand{\X}{\mathbf{X}}
\newcommand{\ban}{\mathrm{Y}}
\newcommand{\id}{\operatorname{id}}
\newcommand{\argmin}{\operatornamewithlimits{argmin}}
\newcommand{\dg}{\W}

\renewcommand{\tilde}{\widetilde}
\renewcommand{\bar}{\overline}

\title[Generalized Fisher information]{The gradient flow of a generalized Fisher information functional with respect to modified Wasserstein distances}
\author[Jonathan Zinsl]{}
\email{zinsl@ma.tum.de}



\begin{document}

\maketitle

\centerline{\scshape Jonathan Zinsl}
\medskip
{\footnotesize
 \centerline{Zentrum f\"ur Mathematik }
   \centerline{Technische Universit\"at M\"unchen}
   \centerline{85747 Garching, Germany}
} 

\bigskip

\begin{abstract}
This article is concerned with the existence of nonnegative weak solutions to a particular fourth-order partial differential equation: it is a formal gradient flow with respect to a generalized Wasserstein transportation distance with nonlinear mobility. The corresponding free energy functional is referred to as \emph{generalized Fisher information functional} since it is obtained by autodissipation of another energy functional which generates the heat flow as its gradient flow with respect to the aforementioned distance. Our main results are twofold: For mobility functions satisfying a certain regularity condition, we show the existence of weak solutions by construction with the well-known minimizing movement scheme for gradient flows. Furthermore, we extend these results to a more general class of mobility functions: a weak solution can be obtained by approximation with weak solutions of the problem with regularized mobility.
\end{abstract}

\section{Introduction}\label{sec:intro}
This work is concerned with the existence of nonnegative weak solutions $u:[0,\infty)\times\Omega\to [0,\infty)$ to the partial differential equation
\begin{align}
\partial_t u(t,x)&=\dv \left(m(u(t,x)) \nabla \frac{\delta \calF}{\delta u}(u(t,x))\right)\label{eq:pde}
\end{align}
for $(t,x)\in(0,\infty) \times \Omega$, where $\Omega\subset\R^d$ $(d\ge 1)$ is a bounded and convex domain with smooth boundary $\partial\Omega$ and exterior unit normal vector field $\nu$. The assumption on the \emph{mobility function} $m$ and the \emph{free energy functional} $\calF$ are specified below. Above, $\frac{\delta \calF}{\delta u}$ denotes the first variation of $\calF$ in $L^2$. Additionally, the sought-for solution $u$ to \eqref{eq:pde} is subject to the no-flux and homogeneous Neumann boundary conditions
\begin{align}
m(u(t,x))\partial_\nu \frac{\delta \calF}{\delta u}(u(t,x))   = 0 &= \partial_\nu u(t,x)  \label{eq:pdebc}
\end{align}
for $t>0$ and $x\in \partial\Omega$, and to the initial condition 
\begin{align}
u(0,\cdot) = u_0\in L^1(\Omega),\text{ with }u_0 \geq 0 \text{ and }  \int_\Omega u_0(x)\dd x=U,\label{eq:ic}
\end{align}
for some fixed $U>0$.\\

Formally, \eqref{eq:pde} possesses a gradient flow structure with respect to a modified version $\W_m$ of the $L^2$-Wasserstein distance on the space of probability measures, reading as
\begin{align}\label{eq:Wm}
\W_m(u_0,u_1)&=\left(\inf_{(u_s,w_s)\in\mathscr{C}}\int_0^1\int_\Omega \frac{|w_s|^2}{m(u_s)}\dd x\dd t\right)^{1/2},
\end{align}
where the admissible set $\mathscr{C}$ is a suitable subclass of curves $(u_s,w_s)_{s\in[0,1]}$ satisfying the \emph{continuity equation} $\partial_s u_s=-\dv\,w_s$ on $[0,1]\times\Omega$ and the initial and terminal conditions $u_s|_{s=0}=u_0$, $u_s|_{s=1}=u_1$. We refer to \cite{dns2009,lisini2010} (see also \cite{carrillo2010,zm2014}) for more details concerning the definition and properties of $\W_m$. 

Before stating our assumptions on the mobility function $m$, let us consider the linear case $m=\id$. Then, it is well-known \cite{brenier2000} that $\W_m$ coincides (up to a scaling factor depending on the value of $U$) with the classical $L^2$-Wasserstein distance $\W_2$ for probability measures on $\Omega$. Using techniques from optimal transportation theory (see, for instance, \cite{villani2003}), various equations of the form \eqref{eq:pde} for $m=\id$ have been interpreted as \emph{gradient systems} with respect to $\W_2$. The specific gradient flow structure often allows for the analysis of the well-posedness of the underlying evolution equation as well as for the study of the qualitative behaviour of the associated gradient flow solutions \cite{jko1998,otto2001}. A powerful tool for those investigations is provided by the so-called \emph{minimizing movement scheme} (cf. \eqref{eq:mms} below), a metric version of the implicit Euler scheme, which is used for the construction of a time-discrete approximative gradient flow. Even in very general situations, this approach can be of use (see, for example, \cite{savare2008} for an abstract description or \cite{blanchet2012,zm2014} in the context of coupled systems). A certain class of equations of the form \eqref{eq:pde} already has successfully been proved to be well-posed by Lisini, Matthes and Savar\'{e} \cite{lisini2012} using its variational structure in spaces w.r.t. the distance $\W_m$. This work aims at a further extension.

In their seminal paper \cite{jko1998} on the variational formulation of the Fokker-Planck equation, Jordan, Kinderlehrer and Otto rigorously interpreted the \emph{heat equation}
\begin{align*}
\partial_t u&=\Delta u
\end{align*}
as $\W_2$-gradient flow of \emph{Boltzmann's entropy}
\begin{align*}
\calH(u)&=\int_\Omega u\log u\dd x.
\end{align*}
Boltzmann's entropy $\calH$ shares an important property \cite{mccann1997}: it is $0$-convex along (generalized) geodesics with respect to the Wasserstein distance $\W_2$ and thus generates a $0$-contractive gradient flow (in the sense of Ambrosio, Gigli and Savar\'{e} \cite{savare2008}) along which $\calH$ descreases most. The corresponding \emph{dissipation} of $\calH$ over time is given by the so-called \emph{Fisher information functional}
\begin{align*}
\calF(u)&=\int_\Omega 4|\nabla\sqrt{u}|^2\dd x.
\end{align*}
Formally, $\calF$ induces the (fourth-order) \emph{Derrida-Lebowitz-Speer-Spohn equation}
\begin{align*}
\partial_t u&=\dv\left(u\nabla\left(\frac{\Delta\sqrt{u}}{\sqrt{u}}\right)\right)
\end{align*}
as $\W_2$-gradient flow. A thorough analysis of the relationship between $\calH$, $\calF$ and their corresponding evolution equations has been done by Gianazza, Savar\'{e} and Toscani \cite{gianazza2009}, later generalized by Matthes, McCann and Savar\'{e} \cite{matthes2009} to more general energy functionals in the Wasserstein framework. Even if the Fisher information functional $\calF$ does not admit a convexity condition along geodesics w.r.t. $\W_2$, results on existence and convergence to equilibrium can be deduced using the fact that $\calF=|\partial\calH|^2$, i.e., being the squared \emph{Wasserstein slope} of Boltzmann's entropy $\calH$. The main aim of this work is to extend this specific connection to nonlinear mobility functions $m$ and the generalized Wasserstein distance $\W_m$.

However, in this case, the structure of the slope w.r.t. $\W_m$ is not known explicitly. Moreover, convexity along geodesics in this space is a very rare property (cf. \cite{carrillo2010,zm2014}), with the following exception: It is known that the \emph{heat entropy functional}
\begin{align*}
\calH(u)=\int_\Omega h(u(x))\dd x,\quad\text{with } h''(z)m(z)=1\text{ for all }z, 
\end{align*}
is $0$-convex along geodesics w.r.t. $\W_m$ and generates the heat flow as its \break $0$-contractive gradient flow. Formally, the dissipation of $\calH$ along its own gradient flow, i.e., its \emph{autodissipation}, reads as
\begin{align*}
-\frac{\dd}{\dd t}\calH(u(t))&=\int_\Omega m(u(t))|\nabla h'(u(t))|^2\dd x=\int_\Omega \frac{|\nabla u(t)|^2}{m(u(t))}\dd x\\&=\int_\Omega \left|\nabla\left(\int^{u(t)}\frac1{\sqrt{m(r)}}\dd r\right)\right|^2\dd x.
\end{align*}
This motivates our specific definition for the free energy functional $\calF$ to be considered in (the consequently fourth-order equation) \eqref{eq:pde}:
\begin{align}\label{eq:fisher}
\calF(u):=\int_\Omega \frac12|\nabla f(u(x))|^2\dd x,\quad\text{with } f(z):=\int_0^z\sqrt{\frac{2}{m(r)}}\dd r,
\end{align}
if $f(u)\in H^1(\Omega)$, and $\calF(u):=+\infty$ otherwise. We call $\calF$ the \emph{generalized Fisher information} functional associated to the mobility $m$. For linear mobility, functionals of the form \eqref{eq:fisher} for other choices of $f$ have already been studied in \cite{lmz2015}.\\

We impose the additional constraint $u(t,x)\le S$, where $S>0$ either is a fixed real number or equal to $+\infty$. The specific value of $S$ is determined by the structure of the mobility function $m$ which is assumed to satisfy the following.
\begin{align}\label{eq:M}
\tag{M}
\begin{split}
&m\in C^2((0,S));\quad m(z)>0\text{ and }m''(z)\le 0\text{ for all }z\in (0,S);\\&m(0)=0,\text{ and, if $S<\infty$, }m(S)=0.
\end{split}
\end{align}
For mobilities satisfying \eqref{eq:M}, one can define the distance $\W_m$ via formula \eqref{eq:Wm} on the space
\begin{align*}
\X:=\left\{u\in L^1(\Omega):~0\le u\le S\text{ a.e. in }\Omega,\,\int_\Omega u\dd x=U\right\},
\end{align*}
where $U>0$ is a fixed given number, see \cite{dns2009,lisini2010}. In this work, we use some of the topological properties of the metric space $(\X,\W_m)$. The respective statements are omitted here for the sake of brevity.\\

In the case $S=\infty$, we need an additional assumption on the growth of $m$ for large $z$:
\begin{align}\label{eq:MPG}
\tag{M-PG}
\begin{split}
&\text{There exist $\gamma_0,\gamma_1\in [0,1]$ with the additional requirement that }\\
&\text{$\frac{2-\gamma_0}{2-\gamma_1}<\frac{d}{d-2}$, if $d\ge 3$, such that:}\\
&\lim_{z\to\infty}\frac{m(z)}{z^{\gamma_0}}\in (0,\infty]\quad\text{and}\quad\lim_{z\to\infty}\frac{m(z)}{z^{\gamma_1}}\in[0,\infty).
\end{split}
\end{align}
An important subclass of mobilities is the following: we say that $m$ satisfies the LSC condition (compare to \cite{lisini2012}) if
\begin{align}\label{eq:MLSC}
\tag{M-LSC}
\begin{split}
\sup_{z\in (0,S)}|m'(z)|<\infty,\text{ and }\sup_{z\in (0,S)}(-m''(z)m(z))<\infty.
\end{split}
\end{align}
In particular, mobilities satisfying \eqref{eq:MLSC} are Lipschitz continuous. In order to also cover the framework of non-LSC mobilities, we will later require the following on the strength of the singularities of $m'$ at the boundary of $(0,S)$, if $m$ does not fulfill \eqref{eq:MLSC}:
\begin{align}
\label{eq:MS}
\tag{M-S}
\begin{split}
&\lim_{z\searrow 0}m'(z)^2 f(z)=0,\\
&\text{ and, if }S<\infty\text{, additionally }\lim_{z\nearrow S}m'(z)^2(f(S)-f(z))=0.
\end{split}
\end{align}
In particular, \eqref{eq:MS} is met by the paradigmatic examples $m(z)=z^\beta$ for $\beta\in(\frac23,1]$ and $S=\infty$ as well as by $m(z)=z^{\beta_1}(S-z)^{\beta_2}$ for $\beta_1,\beta_2\in(\frac23,1]$ and $S<\infty$. Compared to \cite{lisini2012} where $\beta\in (\frac12,1]$ is allowed, we need a slightly stricter condition on $m$ here due to the appearance of $m$ in the definition of $f$.\\

Our proof of existence of solutions to \eqref{eq:pde} with $\calF$ defined as in \eqref{eq:fisher} and mobilities satisfying \eqref{eq:M} and \eqref{eq:MLSC} relies on the formal gradient structure of equation \eqref{eq:pde}. We use the time-discrete \emph{minimizing movement scheme} for the construction of weak solutions in the space $\X$: given a step size $\tau>0$, define a sequence $(u_\tau^n)_{n\ge 0}$ in $\X$ recursively via
\begin{align}\label{eq:mms}
u_\tau^0:=u_0,\qquad u_\tau^n\in\argmin\limits_{u\in\X}\left(\frac1{2\tau}\W_m(u,u_\tau^{n-1})^2+\calF(u)\right)\quad\text{for }n\in\N.
\end{align}

With the sequence $(u_\tau^n)_{n\ge 0}$ from \eqref{eq:mms}, we define a time-discrete function \break $u_\tau:[0,\infty)\times\Omega\to [0,\infty]$ via piecewise constant interpolation:
\begin{align}\label{eq:pwc}
u_\tau(t):=u_\tau^n\qquad\text{if }t\in ((n-1)\tau,n\tau]\text{ for some }n\ge 1;\quad u_\tau(0):=u_0.
\end{align}

The resulting main theorem on the limit behaviour of $(u_{\tau})_{\tau>0}$ in the continuous time limit $\tau\searrow 0$ is as follows.

\begin{theorem}[Existence for Lipschitz mobilities]\label{thm:existLSC}
Assume that \eqref{eq:M} and \eqref{eq:MLSC} hold, and if $S=\infty$, let \eqref{eq:MPG} be satisfied. Let an initial condition $u_0\in\X$ with $\calF(u_0)<\infty$ be given.

Then, for each step size $\tau>0$, a time-discrete function $u_\tau:[0,\infty)\to\X$ can be constructed via the minimizing movement scheme \eqref{eq:mms}\&\eqref{eq:pwc}. Moreover, for each vanishing sequence $\tau_k\to 0$ of step sizes, there exists a (non-relabelled) subsequence and a limit function $u:[0,\infty)\to\X$ such that the following is true for arbitrary $T>0$:
\begin{enumerate}[(a)]
\item $u\in C^{1/2}([0,T];(\X,\W_m))\cap L^\infty([0,T];L^p(\Omega))$ for some $p>1$; and \break $f(u)\in L^\infty([0,T];H^1(\Omega))\cap L^2([0,T];H^2(\Omega))$.
\item $u_{\tau_k}$ converges to $u$ as $k\to\infty$ strongly in $L^p([0,T];L^p(\Omega))$ and pointwise with respect to $t\in [0,T]$ in $(\X,\W_m)$. $f(u_{\tau_k})$ converges to $f(u)$ as $k\to\infty$ strongly in $L^2([0,T];H^1(\Omega))$ and weakly in $L^2([0,T];H^2(\Omega))$.
\item For almost every $t\in[0,T]$, one has $\calF(u_{\tau_k}(t))\to\calF(u(t))$ as $k\to\infty$; and the map $t\mapsto \calF(u(t))$ is almost everywhere equal to a nonincreasing function.
\item The map $u$ is a solution to \eqref{eq:pde}--\eqref{eq:ic} in the sense of distributions.
\end{enumerate}
\end{theorem}
Notice that due to the non-convexity of the problem, we do not obtain uniqueness of solutions. The study of qualitative properties of $u$ with respect to (large) time is postponed to future research.\\

Our second main theorem extends the result to mobilities $m$ which are not Lipschitz continuous. We obtain a weak solution to \eqref{eq:pde} no longer by approximation via the minimizing movement scheme \eqref{eq:mms}, but as a limit of solutions of \eqref{eq:pde} for mobilities $m_\delta$ which are close to $m$ and satisfy \eqref{eq:MLSC}, as $\delta\to 0$. Specifically, we approximate $m$ by $m_\delta$ as in \cite{lisini2012}:\\

\noindent Let $\delta\in (0,\bar\delta)$ and $\bar\delta>0$ sufficiently small.\\
If $S<\infty$, define 
\begin{align}\label{eq:app1}
m_\delta(z):=m\left(\frac{z_{\delta,1}-z_{\delta,2}}{S}z-z_{\delta,1}\right)-\delta,
\end{align}
where $z_{\delta,1}<z_{\delta,2}$ are the two solutions of $m(z)=\delta$.\\
If $S=\infty$, define
\begin{align}\label{eq:app2}
m_\delta(z):=m(z+z_\delta)-\delta,
\end{align}
where $z_\delta$ is the unique solution of $m(z)=\delta$.\\

Since $m_\delta$ is constructed in such a way that \eqref{eq:MLSC} is fulfilled, there exists a weak solution $u_\delta$ to \eqref{eq:pde} with initial condition $u_0$ (which is assumed to be independent of $\delta$ and such that $\calF_\delta(u_0)<\infty$), in the sense of Theorem \ref{thm:existLSC}, for each $\delta$. To indicate the dependence of $\calF$ on $m$, we e.g. write $f_\delta$ and $\calF_\delta$ when $m_\delta$ is considered in place of $m$. Analysing the limit behaviour of the family $(u_\delta)_{\delta>0}$ as $\delta\searrow 0$, we find that the properties of $u_\delta$ carry over to the limit:
\begin{theorem}[Existence for non-Lipschitz mobilities]\label{thm:existnonLSC}
Assume that $m$ satisfies \eqref{eq:M}, \eqref{eq:MS}, and if $S=\infty$, also \eqref{eq:MPG}, and define $m_\delta$ for $\delta\in (0,\bar\delta)$ and sufficiently small $\bar\delta$ as in \eqref{eq:app1}/\eqref{eq:app2}. Let an initial condition $u_0\in\X$ with $\calF_{\bar\delta}(u_0)<\infty$ be given and denote by $u_\delta$ a weak solution to \eqref{eq:pde} with $m_\delta$ in place of $m$ and initial condition $u_0$ in the sense of Theorem \ref{thm:existLSC}, for each $\delta\in (0,\bar\delta)$. Then, there exists a vanishing sequence $\delta_k\to 0$ and a map $u:[0,\infty)\to\X$ such that for the sequence $(u_{\delta_k})_{k\in\N}$ and the limit $u$, one has
\begin{enumerate}[(a)]
\item $u\in C^{1/2}([0,T];(\X,\W_m))\cap L^\infty([0,T];L^p(\Omega))$ for some $p>1$; and \break $f(u)\in L^\infty([0,T];H^1(\Omega))\cap L^2([0,T];H^2(\Omega))$.
\item $u_{\delta_k}$ converges to $u$ as $k\to\infty$ strongly in $L^p([0,T];L^p(\Omega))$ and pointwise with respect to $t\in [0,T]$ in $(\X,\W_m)$. $f_{\delta_k}(u_{\delta_k})$ converges to $f(u)$ as $k\to\infty$ strongly in $L^2([0,T];H^1(\Omega))$ and weakly in $L^2([0,T];H^2(\Omega))$.
\item For almost every $t\in[0,T]$, one has $\calF_{\delta_k}(u_{\delta_k}(t))\to\calF(u(t))$ as $k\to\infty$; and the map $t\mapsto \calF(u(t))$ is almost everywhere equal to a nonincreasing function.
\item The map $u$ is a solution to \eqref{eq:pde}--\eqref{eq:ic} in the sense of distributions.
\end{enumerate}
\end{theorem}

The plan of the paper is as follows. Section \ref{sec:lsc} is concerned with the proof of Theorem \ref{thm:existLSC}. We first study the discrete curves obtained by the minimizing movement scheme in Section \ref{subsec:mms} before passing to the continuous time limit in Section \ref{ssec:ctl}. In Section \ref{sec:nonlsc}, we show Theorem \ref{thm:existnonLSC}. Properties of the approximation with Lipschitz mobilities are investigated in Section \ref{ssec:approx}, its convergence in Section \ref{ssec:conv}

\section{Lipschitz mobility functions}\label{sec:lsc}
In this section, we prove Theorem \ref{thm:existLSC}. In advance of studying the properties of the scheme \eqref{eq:mms}\&\eqref{eq:pwc}, we first prove an auxiliary result on the relationship between $\calF$ and $\calH$. To this end, we make the following specific choice of the integrand $h$ of $\calH$ (compare with \cite{lisini2012}):
\begin{align}\label{eq:hdefi}
\text{For some fixed }s_0\in(0,S),\text{ set }h(z):=\int_{s_0}^z\frac{z-r}{m(r)}\dd r,\quad\text{for }z\in(0,S).
\end{align}
Note that the following statement does not require condition \eqref{eq:MLSC}.
\begin{lemma}[Estimate on $\calH$]\label{lem:HFest}
Assume that $m$ satisfies \eqref{eq:M}, and, if $S=\infty$, also \eqref{eq:MPG}. Then, there exist $C>0$ and $q\ge 1$ such that for all $u\in\X$ with $f(u)\in H^1(\Omega)$, one has 
\begin{align}\label{eq:Hup}
0\le\calH(u)&\le C(\calF(u)^q+1).
\end{align}
\end{lemma}
\begin{proof}
Note that $h(z)\ge 0$ for all $z\in (0,S)$ and $h(s_0)=0$. We first investigate the behaviour of $h$ as $z\searrow 0$: for $z<s_0$, one has
\begin{align*}
h(z)=\int_z^{s_0}\frac{r-z}{m(r)}\dd r&\le \int_z^{s_0}\frac{(r-z)z}{m(z)r}\dd r,
\end{align*}
where the last step follows from concavity of $m$, viz. $m(r)\ge \frac{m(z)}{z}r$. One directly verifies that the limit $\lim\limits_{z\searrow 0}\int_z^{s_0}\frac{(r-z)z}{m(z)r}\dd r$ exists. Thanks to the monotonicity of $h$ for $z< s_0$ (clearly, $h'(z)<0$ for $z<s_0$), also the limit $\lim\limits_{z\searrow 0}h(z)$ exists. If $S<\infty$, existence of the limit $\lim\limits_{z\nearrow S}h(z)$ follows in analogy. Hence, the integrand $h$ can be continuously extended onto the boundary of $(0,S)$. 

So, if $S<\infty$, $h$ is a bounded function. Hence, as $f$ is increasing, we have
\begin{align}\label{eq:hf2}
h(z)\le C(f(z)^2+1),
\end{align}
for some $C>0$ and all $z\in [0,S]$, from which then \eqref{eq:Hup} with $q=1$ follows using Poincar\'{e}'s inequality.

Consider the case $S=\infty$. By assumption \eqref{eq:MPG}, there exist $\tilde z>\max(s_0,1)$ and constants $C_0,C_1>0$ such that for all $z>\tilde z$:
\begin{align}\label{eq:potenz}
C_0z^{\gamma_0}\le m(z)\le C_1z^{\gamma_1}.
\end{align}
In view of the previous arguments on bounded value spaces, we may restrict ourselves to the case $z>\tilde z$ in the following. First, we have
\begin{align*}
h(z)&=\int_{s_0}^{\tilde z}\frac{z-r}{m(r)}\dd r+\int_{\tilde z}^{z}\frac{z-r}{m(r)}\dd r,
\end{align*}
by definition of $h$ \eqref{eq:hdefi}. By similar considerations as above, one easily finds that 
\begin{align*}
\int_{s_0}^{\tilde z}\frac{z-r}{m(r)}\dd r&\le \tilde C z,
\end{align*}
for some $\tilde C>0$. Observe that 
\begin{align}\label{eq:fbelow}
f(z)&\ge C(\sqrt{z}-1),\quad\text{for some }C>0,
\end{align}
which can be verified by elementary calculations, using that $m(z)\le \tilde{C}(z+1)$ as a consequence of assumption \eqref{eq:M}. With \eqref{eq:fbelow}, we again arrive at
\begin{align*}
\int_{s_0}^{\tilde z}\frac{z-r}{m(r)}\dd r&\le C(f(z)^2+1).
\end{align*}
For the second term, we use \eqref{eq:potenz} to obtain:
\begin{align*}
\int_{\tilde z}^{z}\frac{z-r}{m(r)}\dd r&\le \frac1{C_0}\int_{\tilde z}^z\left(zr^{-\gamma_0}-r^{1-\gamma_0}\right)\dd r\\
&=\begin{cases}\frac1{C_0}\left[z(\log z-\log\tilde z-(z-\tilde z))\right]&\text{if }\gamma_0=1,\\ \frac1{C_0}\left[\frac{z}{1-\gamma_0}\left(z^{1-\gamma_0}-\tilde z^{1-\gamma_0}\right)-\frac1{2-\gamma_0}\left(z^{2-\gamma_0}-\tilde z^{2-\gamma_0}\right)\right]&\text{if }\gamma_0\in [0,1).\end{cases}
\end{align*}
All in all, we end up with
\begin{align*}
h(z)\le C(f(z)^2+1)+C'(1+z\log z+z)&\quad\text{if $\gamma_0=1$, and}\\
h(z)\le C(f(z)^2+1)+C'(1+z^{2-\gamma_0})&\quad\text{if $\gamma_0\in [0,1)$.}
\end{align*}
To estimate $f$ from below, we use \eqref{eq:potenz} again:
\begin{align*}
f(z)&=\int_0^{\tilde z}\sqrt{\frac{2}{m(r)}}\dd r+\int_{\tilde z}^z \sqrt{\frac{2}{m(r)}}\dd r\\
&\ge C(\sqrt{\tilde z}-1)+\int_{\tilde z}^z\sqrt{\frac{2}{C_1}}r^{-\frac{\gamma_1}{2}}\dd r\\
&\ge -\tilde C_1+\tilde C_2 z^{1-\frac{\gamma_1}{2}}.
\end{align*}
Consider the case $\gamma_0<1$. If $d\ge 3$, one has by \eqref{eq:MPG} that
\begin{align*}
h(z)&\le C(f(z)^2+1)+C'(1+z^{2-\gamma_0})\\
&\le C(f(z)^2+1)+C'(1+z^{\frac{2d}{d-2}\left(1-\frac{\gamma_1}{2}\right)})\\
&\le  C(f(z)^2+1)+\tilde C(1+f(z)^{\frac{2d}{d-2}}).
\end{align*}
Putting $z=u(x)$ and integrating over $x\in\Omega$, we obtain \eqref{eq:Hup} for some $q\ge 1$ with the Gagliardo-Nirenberg-Sobolev inequality. The cases $d\in\{1,2\}$ and $\gamma_0=1$ can be treated by similar, but easier arguments. 
\end{proof}

\subsection{The minimizing movement scheme}\label{subsec:mms}
The next result shows the well-\break posedness of the scheme \eqref{eq:mms}. Furthermore, the subsequent minimizers $u_\tau^n$ gain in regularity, i.e., not only $f(u_\tau^n)\in H^1(\Omega)$, but even $f(u_\tau^n)\in H^2(\Omega)$ holds.
\begin{proposition}[Properties of the scheme]\label{prop:scheme}
Assume that \eqref{eq:M} and \eqref{eq:MLSC} hold, and if $S=\infty$, let \eqref{eq:MPG} be satisfied. Let an initial condition $u_0\in\X$ with $\calF(u_0)<\infty$ be given. Then, for each $\tau>0$, the scheme \eqref{eq:mms} is well-defined and yields a sequence $(u_\tau^n)_{n\ge 0}$ and a discrete solution $u_\tau$ via \eqref{eq:pwc}. \\ Moreover, the following statements hold:
\begin{enumerate}[(a)]
\item For all $n\in\N$, $s,t\ge 0$, one has:
\begin{align}
\calF(u_\tau^n)&\le \calF(u_\tau^{n-1})\le \calF(u_0)<\infty,\label{eq:enest}\\
\sum_{n=1}^\infty \W_m(u_\tau^n,u_\tau^{n-1})^2&\le 2\tau\calF(u_0),\nonumber\\
\W_m(u_\tau(s),u_\tau(t))&\le \left(2\calF(u_0)\max(|s-t|,\tau)\right)^{1/2}.\label{eq:holdest}
\end{align}
\item There exists a constant $C>0$ such that for all $\tau>0$ and all $n\in\N$, one has:
\begin{align}\label{eq:areg}
\int_\Omega \|\nabla^2 f(u_\tau^n)\|^2\dd x&\le \frac{C}{\tau}\left(\calH(u_\tau^{n-1})-\calH(u_\tau^n)\right).
\end{align}
\item There exists $p>1$ such that for all $T>0$, there exists a constant $C>0$ such that for all $\tau>0$, the following holds:
\begin{align*}
&\|f(u_\tau)\|_{L^\infty([0,T];H^1)}\le C,\\
&\|f(u_\tau)\|_{L^2([0,T];H^2)}\le C,\\
&\|u_\tau\|_{L^\infty([0,T];L^p)}\le C.
\end{align*}
\end{enumerate}
\end{proposition}
\begin{proof}
A straightforward application of the direct method from the calculus of variations and Poincar\'{e}'s inequality shows that, given that $\calF(u_\tau^{n-1})<\infty$, i.e., $f(u_\tau^{n-1})\in H^1(\Omega)$, the \emph{Yosida penalization} $u\mapsto\frac1{2\tau}\W_m(u,u_\tau^{n-1})^2+\calF(u)$ admits a minimizer $u_\tau^n$ on $\X$ with $\calF(u_\tau^n)<\infty$. 

The properties in (a) are a direct consequence of the scheme \eqref{eq:mms}\&\eqref{eq:pwc} and are well-known (see, for instance, \cite{savare2008}).

To prove the additional regularity property \eqref{eq:areg}, we apply the \emph{flow interchange} technique from \cite{matthes2009}, using the heat entropy $\calH$ as $0$-convex auxiliary functional. The relevant symbolic calculations are already contained in the proof of \cite[Prop. 7.3]{lmz2015} where a functional of similar form has been studied: here, it is sufficient to show that
\begin{align*}
\frac{f'''(z)f'(z)}{f''(z)^2}&\ge 3\qquad\text{for all }z\in (0,S).
\end{align*}
Indeed, using the definition of $f$, one sees
\begin{align*}
\frac{f'''(z)f'(z)}{f''(z)^2}&=3-2m(z)\frac{m''(z)}{m'(z)^2},
\end{align*}
and the claim follows by assumption \eqref{eq:M} on the mobility $m$.

The first estimate in (c) is an immediate consequence of the energy estimate \eqref{eq:enest} and Poincar\'{e}'s inequality. The last one is nontrivial only for $d\ge 3$ and $S=\infty$. Using the first estimate and the Gagliardo-Nirenberg-Sobolev inequality, the above estimate with $p=\frac{d}{d-2}$ follows from the inequality \eqref{eq:fbelow}.

For the second statement in (c), integrate \eqref{eq:areg} over time and simplify the telescopic sum to see
\begin{align*}
\int_0^T\int_\Omega \|\nabla^2 f(u_\tau)\|^2\dd x\dd t&\le C\left(\calH(u_0)-\calH(u_\tau^n)\right).
\end{align*}
With \eqref{eq:Hup} and the energy estimate \eqref{eq:enest}, we have for some $C'>0$:
\begin{align*}
\int_0^T\int_\Omega \|\nabla^2 f(u_\tau)\|^2\dd x\dd t&\le C'\left(\calF(u_0)^q+\calF(u_\tau^n)^q+1\right)\le C'\left(2\calF(u_0)^q+1\right),
\end{align*}
which is a finite constant.
\end{proof}
With the techniques from \cite{lisini2012,lmz2015}, one deduces an approximate weak formulation of equation \eqref{eq:pde} satisfied by the discrete curve $u_\tau$. The cornerstone of the derivation again is the \emph{flow interchange lemma} \cite[Thm. 3.2]{matthes2009}. This time, the auxiliary functional is the \emph{regularized potential energy}
\begin{align*}
\mathcal{V}(u):=\beta\calH(u)+\int_\Omega u\phi\dd x,
\end{align*}
where $\beta>0$ and $\phi\in C^\infty(\overline{\Omega})$ with $\partial_\nu\phi=0$ on $\partial\Omega$ are given: it is known \cite{lisini2012,zm2014} under the assumptions \eqref{eq:M} and \eqref{eq:MLSC} on $m$ that $\mathcal{V}$ induces a $\kappa_\beta$-contractive gradient flow on $(\X,\W_m)$, for $\kappa_\beta=-\frac{C}{\beta}$ with a fixed constant $C>0$. For the sake of brevity, we skip the calculations here and directly state the result:
\begin{lemma}[Discrete weak formulation]\label{lem:dweak}
Let $\tau>0$ and define the discrete solution $u_\tau$ by \eqref{eq:mms}\&\eqref{eq:pwc}. Then, for all $\beta>0$, all $\phi\in C^\infty(\overline{\Omega})$ with $\partial_\nu\phi=0$ on $\partial\Omega$ and all $\eta\in C^\infty_c((0,\infty))\cap C([0,\infty))$, the following \emph{discrete weak formulation} holds:
\begin{align}
\begin{split}
&\qquad\|\eta\|_{C^0}\kappa_\beta\tau\calF(u_0)+\beta\int_0^\infty\int_\Omega \frac{|\eta|_\tau(t)-|\eta|_\tau(t+\tau)}{\tau}h(u_\tau)\dd x\dd t\\
&\le \int_0^\infty\int_\Omega \frac{\eta_\tau(t)-\eta_\tau(t+\tau)}{\tau}u_\tau\phi\dd x\dd t\\&\quad+\int_0^\infty\int_\Omega \eta_\tau f'(u_\tau)\Delta f(u_\tau)\left[\nabla m(u_\tau)\cdot\nabla\phi+m(u_\tau)\Delta\phi\right]\dd x \dd t\\
&\le -\|\eta\|_{C^0}\kappa_\beta\tau\calF(u_0)-\beta\int_0^\infty\int_\Omega \frac{|\eta|_\tau(t)-|\eta|_\tau(t+\tau)}{\tau}h(u_\tau)\dd x\dd t,
\end{split}
\label{eq:dweak}
\end{align}
where $\kappa_\beta<0$ is as above. For $a:[0,\infty)\to \R$, we denote $a_\tau(s)=a(\lceil\frac{s}{\tau}\rceil\tau)$.
\end{lemma}
Note that since $f'(z)=\sqrt{\frac{2}{m(z)}}$ holds, one can also rewrite the term involving $f$ in \eqref{eq:dweak} as
\begin{align*}
&\quad\int_0^\infty\int_\Omega \eta_\tau f'(u_\tau)\Delta f(u_\tau)\left[\nabla m(u_\tau)\cdot\nabla\phi+m(u_\tau)\Delta\phi\right]\dd x \dd t\\
&=\int_0^\infty\int_\Omega \eta_\tau \sqrt{2}\Delta f(u_\tau)\left[2\nabla \sqrt{m(u_\tau)}\cdot\nabla\phi+\sqrt{m(u_\tau)}\Delta\phi\right]\dd x \dd t.
\end{align*}
\subsection{Passage to the continuous time limit}\label{ssec:ctl}
The remainder of this section is concerned with the passage to the continuous time limit $\tau\searrow 0$. In particular, we show that the passage to the limit inside \eqref{eq:dweak} yields the time-continuous weak formulation of \eqref{eq:pde}, completing the proof of Theorem \ref{thm:existLSC}. In order to obtain convergence in a strong sense, we make use of the following extension of the Aubin-Lions compactness lemma.
\begin{theorem}[Extension of the Aubin-Lions lemma {\cite[Thm. 2]{rossi2003}}]\label{thm:ex_aub}
Let $\ban$ be a Banach space and $\mathcal{A}:\,\ban\to[0,\infty]$ be lower semicontinuous and have relatively compact sublevels in $\ban$. Let furthermore $\dg:\,\ban\times \ban\to[0,\infty]$ be lower semicontinuous and such that $\dg(u,\tilde u)=0$ for $u,\tilde u\in \operatorname{Dom}(\mathcal{A})$ implies $u=\tilde u$.

Let $(U_k)_{k\in\N}$ be a sequence of measurable functions $U_k:\,(0,T)\to \ban$. If
\begin{align*}
&\sup_{k\in\N}\int_0^T\mathcal{A}(U_k(t))\dd t<\infty,\\
&\lim_{h\searrow 0}\sup_{k\in\N}\int_0^{T-h}\dg(U_k(t+h),U_k(t))\dd t=0,
\end{align*}
then there exists a subsequence that converges in measure w.r.t. $t\in(0,T)$ to a limit $U:\,(0,T)\to \ban$.
\end{theorem}
\noindent\textit{Proof of Theorem \ref{thm:existLSC}:}
Let a vanishing sequence $(\tau_k)_{k\in\N}$ of step sizes be given and define the corresponding sequence of discrete solutions $(u_{\tau_k})$ via \eqref{eq:mms}\&\eqref{eq:pwc}. Thanks to the \emph{a priori} estimates from Proposition \ref{prop:scheme}(a)--(c), there exists a map \break $u\in C^{1/2}([0,T];(\X,\W_m))\cap L^\infty([0,T];L^p(\Omega))$ such that $u_{\tau_k}$ converges (on a non-relabelled subsequence) to $u$ both weakly in $L^p([0,T];L^p(\Omega))$ (as a consequence of the Banach-Alaoglu theorem) as well as in $(\X,\W_m)$ pointwise w.r.t. $t\in[0,T]$ (as a consequence of the topological properties of the distance $\W_m$ and a refined version of the Arzel\`{a}-Ascoli theorem \cite[Prop. 3.3.1]{savare2008}). 

Furthermore, Proposition \ref{prop:scheme}(c) and the Banach-Alaoglu theorem yield the existence of $v\in L^2([0,T];H^2(\Omega))$ such that $f(u_{\tau_k})$ converges weakly to $v$ in \break $L^2([0,T];H^2(\Omega))$, possibly extracting another subsequence.  We now prove that $f(u_{\tau_k})\to f(u)$ strongly in $L^2([0,T];L^2(\Omega))$ for $u=f^{-1}\circ v$. By a standard interpolation inequality, the desired strong convergence of $f(u_{\tau_k})\to f(u)$ in $L^2([0,T];H^1(\Omega))$ then follows. Strong convergence of $u_{\tau_k}$ to $u$ in $L^p([0,T];L^p(\Omega))$ (on a subsequence) is achieved by essentially the same technique as in \cite{zm2014,lmz2015} applying Theorem \ref{thm:ex_aub} for the admissible choices $\ban:=L^p(\Omega)$ (with $p>1$ from Proposition \ref{prop:scheme}(c)), 
\begin{align*}
\mathcal{A}:\ban\to[0,\infty]\text{ with }\mathcal{A}(\rho):=\begin{cases}\|\rho\|_{H^2}^2&\text{if }\rho\in H^2(\Omega),\\ +\infty&\text{otherwise,}\end{cases}
\end{align*}
and
\begin{align*}
\dg:\ban\times\ban\to[0,\infty]\text{ with }\dg(\rho,\tilde\rho):=\begin{cases}\W_m(\rho,\tilde\rho)&\text{if }\rho,\tilde\rho\in\X,\\ +\infty&\text{otherwise.}\end{cases}
\end{align*}
We refer to \cite{lmz2015} for the details. A rather straightforward application of Vitali's convergence theorem subsequently yields the strong convergence of $f(u_{\tau_k})\to f(u)$ in $L^2([0,T];L^2(\Omega))$. 

Obviously, this convergence property shows in combination with \eqref{eq:enest} the claimed convergence and monotonicity of the energy $\calF$ in Theorem \ref{thm:existLSC}(c). To prove that the limit $u$ indeed is a weak solution to \eqref{eq:pde}, we verify that for all $\phi\in C^\infty(\overline{\Omega})$ with $\partial_\nu\phi=0$ on $\partial\Omega$ and all $\eta\in C^\infty_c((0,\infty))\cap C([0,\infty))$, the following \emph{continuous weak formulation} holds:
\begin{align}\label{eq:weak}
0=\int_0^\infty\int_\Omega\left( -\partial_t\eta\phi u+\eta f'(u)\Delta f(u)\left[\nabla m(u)\cdot\nabla\phi+m(u)\Delta\phi\right]\right)\dd x \dd t.
\end{align}
We set $\beta_k:=\sqrt{\tau_k}$ in the discrete weak formulation \eqref{eq:dweak}. Then, as $(\calH(u_\tau^n))_{n\in\N}$ is bounded (recall \eqref{eq:Hup}\&\eqref{eq:enest}), it is immediate that one has
\begin{align*}
&\quad\int_0^\infty\int_\Omega\partial_t\eta\phi u\dd x\dd t\\
&=\lim_{k\to\infty}\int_0^\infty\int_\Omega \eta_{\tau_k} \sqrt{2}\Delta f(u_{\tau_k})\left[2\nabla \sqrt{m(u_{\tau_k})}\cdot\nabla\phi+\sqrt{m(u_{\tau_k})}\Delta\phi\right]\dd x \dd t.
\end{align*}
Since $\eta_{\tau_k}\to\eta$ uniformly and $\Delta f(u_{\tau_k})\rightharpoonup\Delta f(u)$ weakly in $L^2([0,T];L^2(\Omega))$, it suffices to prove that $\nabla \sqrt{m(u_{\tau_k})}\to\nabla \sqrt{m(u)}$ strongly in $L^2([0,T];L^2(\Omega))$, in view of Poincar\'{e}'s inequality. Using the definition of $f$ and writing $g:=f^{-1}$ (which exists by strict monotonicity, recall that $m(z)>0$ for $z\in (0,S)$), this is equivalent to proving that $\nabla g'(v_{\tau_k})$ converges to $\nabla g'(v)$ strongly in $L^2([0,T];L^2(\Omega))$ for $v:=f(u)$ and $v_{\tau_k}:=f(u_{\tau_k})$, since one has
\begin{align}\label{eq:dg}
g'(w)=\frac1{f'(g(w))}=\sqrt{\frac12 m(g(w))}\quad\text{for all }w\in (0,f(S)).
\end{align}
Using the chain rule and \eqref{eq:dg}, we obtain
\begin{align*}
g''(w)=\frac14 m'(g(w)).
\end{align*}
Hence, for all $k\in\N$:
\begin{align}\label{eq:nabdg}
\nabla g'(v_{\tau_k})=\frac14 m'(g(v_{\tau_k}))\nabla v_{\tau_k}.
\end{align}
Since $m$ satisfies \eqref{eq:M} and \eqref{eq:MLSC}, $m'\circ g$ is continuous and bounded. Due to the strong convergence of $v_{\tau_k}$ to $v$ in $L^2([0,T];H^1(\Omega))$ and Vitali's theorem, $(\nabla v_{\tau_k})_{k\in\N}$ is uniformly integrable in $L^2([0,T]\times\Omega)$. Without restruction, we may assume that $v_{\tau_k}$ to $v$ almost everywhere on $[0,T]\times\Omega$. Hence, another application of Vitali's theorem yields the asserted strong convergence of $\frac14 m'(g(v_{\tau_k}))\nabla v_{\tau_k}$ to $\frac14 m'(g(v))\nabla v$ in $L^2([0,T]\times\Omega)$, on a suitable subsequence.
\hfill\qed

\section{Non-Lipschitz mobility functions}\label{sec:nonlsc}
In this section, we consider mobility functions which do not satisfy \eqref{eq:MLSC}, but can be approximated in a suitable way by LSC mobilities, see \eqref{eq:app1}\&\eqref{eq:app2}. Our strategy of proof for Theorem \ref{thm:existnonLSC} is as follows: first, we demonstrate that the \emph{a priori} estimates from Proposition \ref{prop:scheme} are \emph{uniform} w.r.t. the approximation parameter $\delta>0$ when considering a family $(u_\delta)_{\delta\in(0,\bar\delta)}$ of weak solutions to \eqref{eq:pde} with initial condition $u_0$ (independent of $\delta$) for $m_\delta$ in place of $m$ (in the sense of Theorem \ref{thm:existLSC}). This will allow us to pass to the limit $\delta\searrow 0$ in the weak formulation \eqref{eq:weak} of \eqref{eq:pde} for $m_\delta$ to obtain the sought-for weak formulation of \eqref{eq:pde} for $m$.

\subsection{A priori estimates}\label{ssec:approx}
This section is devoted to the derivation of the necessary \emph{a priori} estimates on the family $(u_\delta)_{\delta\in(0,\bar\delta)}$ introduced above. For definiteness, notice that 
at each fixed $u\in\X$ and for all $0<\delta_0\le\delta_1\le\bar\delta$, one has
\begin{align*}
\calF_{\delta_0}(u)&=\int_\Omega \eins_{\{y\in\Omega:\,u(y)\in (0,S)\}}(x)f_{\delta_0}'(u(x))^2|\nabla u(x)|^2\dd x\\
&=\int_\Omega \eins_{\{y\in\Omega:\,u(y)\in (0,S)\}}(x)\frac{2}{m_{\delta_0}(u(x))^2}|\nabla u(x)|^2\dd x\\
&\le \int_\Omega \eins_{\{y\in\Omega:\,u(y)\in (0,S)\}}(x)\frac{2}{m_{\delta_1}(u(x))^2}|\nabla u(x)|^2\dd x=\calF_{\delta_1}(u),
\end{align*}
since $m(z)\ge m_{\delta_0}(z)\ge m_{\delta_1}(z)$ for all $z\in (0,S)$ (see \cite{lisini2012}). Hence,
\begin{align*}
\calF_\delta(u_0)\le\calF_{\bar\delta}(u_0),
\end{align*}
so the condition $\calF_{\bar\delta}(u_0)<\infty$ and Theorem \ref{thm:existLSC} provide the existence of $(u_\delta)_{\delta\in(0,\bar\delta)}$.
\begin{lemma}[A priori estimates]\label{lem:apri2}
Let a sufficiently small $\bar\delta>0$ be given and let $(u_\delta)_{\delta\in(0,\bar\delta)}$ be a family of weak solutions to \eqref{eq:pde} for $m_\delta$ in place of $m$ with initial condition $u_0$ (in the sense of Theorem \ref{thm:existLSC}). Then, for all $\delta\in(0,\bar\delta)$ and all $T>0$:
\begin{enumerate}[(a)]
\item $\calF_\delta(u_\delta(t))\le \calF_{\bar\delta}(u_0)$ for all $t\in[0,T]$,
\item $\W_m(u_\delta(s),u_\delta(t))\le \sqrt{2\calF_{\bar\delta}(u_0)|s-t|}$ for all $s,t\in [0,T]$,
\item $\|f_\delta(u_\delta)\|_{L^\infty([0,T];H^1)}\le C$,
\item $\|f_\delta(u_\delta)\|_{L^2([0,T];H^2)}\le C$,
\item $\|u_\delta\|_{L^\infty([0,T];L^p)}\le C$,
\end{enumerate}
for some $p>1$ and $C>0$ independent of $\delta$.
\end{lemma}
\begin{proof}
For part (a), we obtain by Theorem \ref{thm:existLSC}(c):
\begin{align*}
\calF_\delta(u_\delta(t))&\le \calF_\delta(u_0)\le \calF_{\bar\delta}(u_0)<\infty\quad\text{ for all $t\in[0,T]$.}
\end{align*}
Consequently, (b) immediately follows from the H\"older estimate for $u_{\delta}$ in $(\X,\W_{m_\delta})$ (recall \eqref{eq:holdest} and Theorem \ref{thm:existLSC}(b)) and the monotonicity $\W_{m}(u,\tilde u)\le \W_{m_\delta}(u,\tilde u)$ for each $u,\tilde u\in\X$. The claims (c)--(e) are a consequence of (a) and the respective estimates of Proposition \ref{prop:scheme}(c) the proof of which does not rely on condition \eqref{eq:MLSC}.
\end{proof}

\subsection{Convergence}\label{ssec:conv}
In this section, we prove Theorem \ref{thm:existnonLSC}. As a preparation, we show
\begin{lemma}[Local uniform convergence]\label{lem:locuni}
Let a vanishing sequence $(\delta_k)_{k\in\N}$ in $(0,\bar\delta)$ be given and denote, for each $k$, $g_{\delta_k}:=f_{\delta_k}^{-1}$. The following statements hold:
\begin{enumerate}[(a)]
\item If $S=\infty$, there exists a (non-relabelled) subsequence on which the sequence $(G_{\delta_k})_{k\in\N}$, defined by
\begin{align*}
G_{\delta_k}&:[0,\infty)\to\R,\quad G_{\delta_k}(w):=m'_{\delta_k}(g_{\delta_k}(w))\sqrt{w},
\end{align*}
converges locally uniformly to the continuous map
\begin{align*}
G&:[0,\infty)\to\R,\quad G(w):=m'(g(w))\sqrt{w}\text{ for }w>0,\quad G(0):=0.
\end{align*}
\item If $S<\infty$, there exists a (non-relabelled) subsequence on which the sequence $(\tilde G_{\delta_k})_{k\in\N}$, defined by
\begin{align*}
\tilde G_{\delta_k}&:[0,f(S)]\to\R,\quad \tilde G_{\delta_k}(w):=m'_{\delta_k}(g_{\delta_k}(w))\sqrt{w(f_{\delta_k}(S)-w)},
\end{align*}
converges uniformly to the continuous map
\begin{align*}
\tilde G:[0,f(S)]\to\R,\quad \tilde G(w)&:=m'(g(w))\sqrt{w(f(S)-w)}\text{ for }w\in(0,f(S)),\\ 
\tilde G(0)&:=0,\quad \tilde G(f(S)):=0.
\end{align*}
\end{enumerate}
\end{lemma}
\begin{proof}
At first, we prove---for arbitrary $S$---that $(g_{\delta_k})_{k\in\N}$ converges (on a suitable subsequence) locally uniformly on $\overline{[0,S)}$ to $g:=f^{-1}$. Indeed, using the monotonicity $m_{\delta_k}\le m$ on $\overline{[0,S)}$ and \eqref{eq:dg}, we obtain the differential estimate
\begin{align*}
0\le g_{\delta_k}'(w)=\sqrt{\frac12 m_{\delta_k}(g_{\delta_k}(w))}\le \sqrt{\frac12 m(g_{\delta_k}(w))}\le C(g_{\delta_k}(w)+1),
\end{align*}
where the constant $C>0$ does not depend on $k$.
Using Gronwall's lemma, we deduce that $g_{\delta_k}$, and consequently also $g_{\delta_k}'$, is $k$-uniformly bounded on compact subsets of $\overline{[0,S)}$. The application of the Arzel\`{a}-Ascoli theorem yields the desired local uniform convergence.

Consider the case $S=\infty$. Using the monotonicity properties (see \cite{lisini2012} again)
\begin{align}\label{eq:mono}
m_{\delta_k}'(z)\le m'(z)\quad\text{and}\quad g(w)\ge g_{\delta_k}(w)
\end{align}
in combination with the concavity of $m$, we find that for all $w>0$:
\begin{align*}
m_{\delta_k}'(g(w))\sqrt{w}\le G_{\delta_k}(w)\le m'(g_{\delta_k}(w))\sqrt{w}.
\end{align*}
Clearly, $G_{\delta_k}(w)\to G(w)$ if $w>0$. Moreover, this property also extends to $w=0$ by existence of the limit as $w\searrow 0$, cf. condition \eqref{eq:MS}, and $G$ is continuous. Observe that $(m_{\delta_k}'(g(w))\sqrt{w})_{k\in\N}$ and $(m'(g_{\delta_k}(w))\sqrt{w})_{k\in\N}$ are monotonic in $k$ at each fixed $w>0$. Invoking Dini's theorem and a diagonal argument, we deduce the claimed local uniform convergence of $G_{\delta_k}$ to $G$, extracting a certain subsequence.

Consider now the case $S<\infty$. Fix $0<\tilde z_0<\tilde z_1<S$ such that $m'(\tilde z_0)>0$ and $m'(\tilde z_1)<0$, and define $\tilde w_0=f(\tilde z_0)$, $\tilde w_1=f(\tilde z_1)$. We distinguish the cases $w\in [0,\tilde w_0]$, $w\in [\tilde w_0,\tilde w_1]$ and $w\in [\tilde w_1,f(S)]$. Clearly, by smoothness of $m$ and $m_{\delta_k}$, the claim is true for $w\in [\tilde w_0,\tilde w_1]$. For the case $w\in [0,\tilde w_0]$, we argue similarly as above; recall that $f_{\delta_k}(S)\downarrow f(S)$ and observe that \eqref{eq:mono} also holds in the case at hand:
\begin{align*}
m'_{\delta_k}(g(w))\sqrt{w(f(S)-w)}\le \tilde G_{\delta_k}(w)\le m'(g_{\delta_k}(w))\sqrt{w(f_{\delta_k}(S)-w)}.
\end{align*}
Again, $\tilde G_{\delta_k}(w)\to \tilde G(w)$ if $w\ge 0$, and $(m'_{\delta_k}(g(w))\sqrt{w(f(S)-w)})_{k\in\N}$ and \break $(m'(g_{\delta_k}(w))\sqrt{w(f_{\delta_k}(S)-w)})_{k\in\N}$ are monotonic, so proceed as above. For the remaining case $w\in [\tilde w_1,f(S)]$, we have:
\begin{align*}
|\tilde G_{\delta_k}(w)|\le |m'(g_{\delta_k}(w))|\sqrt{w(f_{\delta_k}(S)-w)}\le |\tilde G(w)|\sqrt\frac{f_{\delta_k}(S)-w}{f(S)-w}\le 2|\tilde G(w)|,
\end{align*}
provided that $k$ is sufficiently large. Consequently, by \eqref{eq:MS}, there exists for each $\ep>0$ a $\tilde w\ge \tilde w_1$ such that $|\tilde G_{\delta_k}(w)-\tilde G(w)|\le 2|\tilde G(w)|<\ep$ for all $w\in[\tilde w,f(S)]$ and sufficiently large $k\in\N$. Combining this with the---again obvious---uniform convergence of $\tilde G_{\delta_k}(w)$ to $\tilde G(w)$ on $[\tilde w_1,\tilde w]$ yields the claim.
\end{proof}

Now, we are in position to complete the proof of Theorem \ref{thm:existnonLSC}.\\

\noindent\textit{Proof of Theorem \ref{thm:existnonLSC}:}
In complete analogy to the proof of Theorem \ref{thm:existLSC} from Section \ref{ssec:ctl}, one deduces the existence of a limit map $u:[0,\infty)\to\X$ such that, on a suitable subsequence of the vanishing sequence $(\delta_k)_{k\in\N}$ in $(0,\bar\delta)$, the claims (a)--(c) in Theorem \ref{thm:existnonLSC} are true. It remains to show that the limit $u$ satisfies the weak formulation \eqref{eq:weak}, given that $u_{\delta_k}$ satisfies \eqref{eq:weak} for $m_{\delta_k}$ and $f_{\delta_k}$ in place of $m$ and $f$, respectively. As before, the proof is complete if, extracting a further subsequence, $\nabla \sqrt{m_{\delta_k}(u_{\delta_k})}\to\nabla \sqrt{m(u)}$ strongly in $L^2([0,T];L^2(\Omega))$. Again, introducing the inverse functions $g_{\delta_k}:=f^{-1}_{\delta_k}$ and $g:=f^{-1}$, we have to verify that $m'(g_{\delta_k}(v_{\delta_k}))\nabla v_{\delta_k}\to m'(g(v))\nabla v$ strongly in $L^2([0,T];L^2(\Omega))$, for $v:=f(u)$ and $v_{\delta_k}:=f_{\delta_k}(u_{\delta_k})$ (recall \eqref{eq:dg}\&\eqref{eq:nabdg}).

We first consider the case $S=\infty$ and observe that 
\begin{align*}
m'(g_{\delta_k}(v_{\delta_k}))\nabla v_{\delta_k}=2G_{\delta_k}(v_{\delta_k})\nabla \sqrt{v_{\delta_k}}\qquad\text{a.e. on }[0,T]\times\Omega.
\end{align*}
By Lemma \ref{lem:locuni} and since, without restriction, $\nabla v_{\delta_k}\to \nabla v$ pointwise almost everywhere on $[0,T]\times\Omega$, we obtain that
\begin{align*}
2G_{\delta_k}(v_{\delta_k})\nabla \sqrt{v_{\delta_k}}\to 2G(v)\nabla\sqrt{v}=m'(g(v))\nabla v
\end{align*}
pointwise almost everywhere on $[0,T]\times\Omega$. Furthermore, using \eqref{eq:MS} and \eqref{eq:M} yields $G_{\delta_k}(v_{\delta_k})\le C(\sqrt{w}+1)$ for some $k$-independent constant $C>0$. Hence, for each measurable set $A\subset [0,T]\times\Omega$, one has for sufficiently large $k$ that
\begin{align*}
\int_A \left[2G_{\delta_k}(v_{\delta_k})\nabla \sqrt{v_{\delta_k}}\right]^2 \dd x\dd t&\le 4 \left(\int_0^T \|\nabla\sqrt{v_{\delta_k}}\|_{L^4}^4\dd t\right)^{1/2}\left(\int_A G_{\delta_k}(v_{\delta_k})^4\dd x\dd t\right)^{1/2}\\
&\le C' \left(\int_0^T \|\nabla^2v_{\delta_k}\|_{L^2}^2\dd t\right)^{1/2}\left(\int_A (v_{\delta_k}^2+1) \dd x\dd t\right)^{1/2},
\end{align*}
where we used the well-known \emph{Lions-Villani estimate on square roots} \cite{lions1995} (see \cite[Lemma A.1]{lisini2012} for a formulation in the framework at hand) in the last step. Since $(v_{\delta_k})_{k\in\N}$ is $L^2$-uniformly integrable (by Vitali's theorem), the former estimate also yields $L^2$-uniform integrability of $(2G_{\delta_k}(v_{\delta_k})\nabla \sqrt{v_{\delta_k}})_{k\in\N}$ as $v_{\delta_k}$ is $k$-uniformly \break bounded in $L^2([0,T];H^2(\Omega))$. Applying Vitali's theorem once again gives \break $m'(g_{\delta_k}(v_{\delta_k}))\nabla v_{\delta_k}\to m'(g(v))\nabla v$ strongly in $L^2([0,T];L^2(\Omega))$, extracting a subsequence if necessary.

Consider the remaining case $S<\infty$ and notice that
\begin{align*}
m'(g_{\delta_k}(v_{\delta_k}))\nabla v_{\delta_k}=\frac2{f_{\delta_k}(S)}\tilde G_{\delta_k}(v_{\delta_k})\left[\nabla \sqrt{v_{\delta_k}}+\nabla\sqrt{f_{\delta_k}(S)-v_{\delta_k}}\right]
\end{align*}
almost everywhere on $[0,T]\times\Omega$. Thanks to Lemma \ref{lem:locuni}, one has
\begin{align*}
&\quad\frac2{f_{\delta_k}(S)}\tilde G_{\delta_k}(v_{\delta_k})\left[\nabla \sqrt{v_{\delta_k}}+\nabla\sqrt{f_{\delta_k}(S)-v_{\delta_k}}\right]\\
&\to \frac2{f(S)}\tilde G(v)\left[\nabla \sqrt{v}+\nabla\sqrt{f(S)-v}\right]=m'(g(v))\nabla v
\end{align*}
almost everywhere on $[0,T]\times\Omega$. Now, similarly as above, for each measurable set $A\subset [0,T]\times\Omega$:
\begin{align*}
&\quad\int_A \left(\frac2{f_{\delta_k}(S)}\tilde G_{\delta_k}(v_{\delta_k})\left[\nabla \sqrt{v_{\delta_k}}+\nabla\sqrt{f_{\delta_k}(S)-v_{\delta_k}}\right]\right)^2\dd x\dd t\\
&\le \tilde C' \left(\int_0^T \|\nabla^2v_{\delta_k}\|_{L^2}^2\dd t\right)^{1/2}\left(\int_A (v_{\delta_k}^2+1) \dd x\dd t\right)^{1/2},
\end{align*}
for some $\tilde C'>0$ which does not depend on $k$. Applying Vitali's theorem as in the former case yields the asserted strong convergence $m'(g_{\delta_k}(v_{\delta_k}))\nabla v_{\delta_k}\to m'(g(v))\nabla v$ in $L^2([0,T];L^2(\Omega))$.

All in all, we have proved that $u=f(v)$ satisfies the weak formulation \eqref{eq:weak} of \eqref{eq:pde}, so the proof of Theorem \ref{thm:existnonLSC} is finished.
\qed


\bibliographystyle{AIMS}
   \bibliography{fbib}


\end{document}